\documentstyle[amscd,amssymb,xypic,verbatim,11pt]{amsart}
\xyoption{all}

\newcommand{\nc}{\newcommand}

\nc{\lan}{\big\langle}
\nc{\ran}{\big\rangle}

\nc{\kk}{{\mathsf{k}}}

\nc{\LL}{{\mathbb{L}}}
\nc{\PP}{{\mathbb{P}}}
\nc{\QQ}{{\mathbb{Q}}}
\nc{\RR}{{\mathbb{R}}}
\nc{\TT}{{\mathbb{T}}}
\nc{\ZZ}{{\mathbb{Z}}}

\nc{\CA}{{\mathcal{A}}}
\nc{\CB}{{\mathcal{B}}}
\nc{\D}{{\mathcal{D}}}
\nc{\CE}{{\mathcal{E}}}
\nc{\CM}{{\mathcal{M}}}
\nc{\CO}{{\mathcal{O}}}
\nc{\CS}{{\mathcal{S}}}
\nc{\CT}{{\mathcal{T}}}

\nc{\BP}{{\mathbf{P}}}

\nc{\TBP}{{\tilde{\BP}}}

\nc{\TD}{{\widetilde{\D}}}
\nc{\TCA}{{\tilde{\CA}}}
\nc{\TY}{{\widetilde{Y}}}

\nc{\RHom}{\mathop{\mathsf{RHom}}\nolimits}
\nc{\Hom}{\mathop{\mathsf{Hom}}\nolimits}
\nc{\Ext}{\mathop{\mathsf{Ext}}\nolimits}
\nc{\Tor}{\mathop{\mathsf{Tor}}\nolimits}
\nc{\Pic}{\mathop{\mathsf{Pic}}\nolimits}
\renewcommand{\Im}{\mathop{\mathsf{Im}}\nolimits}
\nc{\Br}{\mathop{\mathsf{Br}}\nolimits}
\nc{\Cone}{\mathop{\mathsf{Cone}}\nolimits}
\nc{\codim}{\mathop{\mathsf{codim}}\nolimits}
\nc{\sing}{{\mathsf{sing}}}
\nc{\perf}{{\mathsf{perf}}}
\nc{\Pf}{{\mathsf{Pf}}}
\nc{\Gr}{{\mathsf{Gr}}}
\nc{\GL}{{\mathsf{GL}}}
\nc{\PGL}{{\mathsf{PGL}}}
\nc{\ch}{{\mathsf{ch}}}
\nc{\td}{{\mathsf{td}}}
\nc{\id}{{\mathsf{id}}}

\theoremstyle{plain}

\newtheorem{theorem}{Theorem}[section]
\newtheorem{conjecture}[theorem]{Conjecture}
\newtheorem{lemma}[theorem]{Lemma}
\newtheorem{proposition}[theorem]{Proposition}
\newtheorem{corollary}[theorem]{Corollary}

\theoremstyle{definition}

\newtheorem{definition}[theorem]{Definition}

\theoremstyle{remark}

\newtheorem{remark}[theorem]{Remark}

\newenvironment{proof}{\noindent{\sf Proof:}}{\qed\medskip}
\newenvironment{proofof}[1]{\noindent{\sf #1:}}{\qed\medskip}

\title{Derived categories of cubic fourfolds}
\author{Alexander Kuznetsov}
\address{\sloppy
\parbox{0.9\textwidth}{
Algebra Section, Steklov Mathematical Institute,
8 Gubkin str., Moscow 119991 Russia
\hfill\\[5pt]
The Poncelet Laboratory, Independent University of Moscow
\hfill\\[5pt]
}}
\email{akuznet@@mi.ras.ru}
\date{}
\thanks{I was partially supported by
RFFI grants 05-01-01034, 07-01-00051 and 07-01-92211,
INTAS 05-1000008-8118,
the Russian Science Support Foundation,
and gratefully acknowledge of the support of the Pierre Deligne fund based on his 2004
Balzan prize in mathematics.}

\begin{document}

\begin{abstract}
We discuss the structure of the derived category of coherent sheaves
on cubic fourfolds of three types: Pfaffian cubics, cubics containing
a plane and singular cubics, and discuss its relation to the rationality
of these cubics.
\end{abstract}

\maketitle

\section{Introduction}


In this paper we are going to discuss one of the classical problems of birational
algebraic geometry, the problem of rationality of a generic cubic hypersurface
in $\PP^5$. Our point of view will be somewhat different from the traditional
approaches. We will use the derived category of coherent sheaves on the
cubic hypersurface (more precisely, a certain piece of this category)
as an indicator of nonrationality.

To be more precise, let $V$ be a vector space of dimension 6, so that $\PP(V) = \PP^5$.
Let $Y \subset \PP(V)$ be a hypersurface of degree $3$,
a cubic fourfold.
By Lefschetz hyperplane section theorem $\Pic Y = \ZZ$
and is generated by $H$, the restriction of the class of a hyperplane in $\PP(V)$.
By adjunction $K_Y = -3H$. So, $Y$ is a Fano variety.
By Kodaira vanishing we can easily compute the cohomology of line bundles $\CO_Y$, $\CO_Y(-1)$ and $\CO_Y(-2)$.
\begin{equation}\label{cohoy}
\dim H^p(Y,\CO_Y(-t)) = \begin{cases}
1, & \text{for $p = t = 0$}\\
0, & \text{for $-2 \le t \le 0$ and $(p,t) \ne (0,0)$}
\end{cases}
\end{equation}
As a consequence, we see that the triple $(\CO_Y, \CO_Y(1), \CO_Y(2))$ is an exceptional collection
in $\D^b(Y)$, the bounded derived category of coherent sheaves on $Y$.
We denote by $\CA_Y$ the orthogonal subcategory to this exceptional collection:
\begin{equation}\label{defay}
\hspace{-2em}
\CA_Y = \langle \CO_Y, \CO_Y(1), \CO_Y(2) \rangle^\perp =
\{ F \in \D^b(Y)\ |\ H^\bullet(Y,F) = H^\bullet(Y,F(-1)) = H^\bullet(Y,F(-2)) = 0 \}.
\hspace{-2em}
\end{equation}
Then we have a semiorthogonal decomposition
\begin{equation}\label{dby}
\D^b(Y) = \langle \CA_Y, \CO_Y, \CO_Y(1), \CO_Y(2) \rangle
\end{equation}
(see Section~\ref{secpre} for the definition of semiorthogonal decompositions).

This triangulated category $\CA_Y$ is the piece of $\D^b(Y)$ discussed above.
By many features it looks like the derived category of a K3 surface
(for example, its Serre functor equals the shift by two functor $[2]$,
and its Hochschild homology is very similar to that of a K3 surface).
Moreover, as we shall see, for some cubic fourfolds $\CA_Y$ is equivalent
to the derived category of a K3 surface. We expect that this happens if and only if $Y$ is rational.

\begin{conjecture}\label{4dconj}
The cubic fourfold $Y$ is rational if and only if the subcategory $\CA_Y \subset \D^b(Y)$ defined by~\eqref{defay}
is equivalent to the derived category of a K3 surface.
\end{conjecture}

The ``only if'' part of the above is a special case of a more general Conjecture suggested in~\cite{K4},
where a new birational invariant, the Clemens--Griffiths component of the derived category, is constructed.
While it is very hard (and not completely clear how) to check that the definition of this component is correct,
it is more or less straightforward that it is preserved under smooth blow-ups, which implies that it is
a birational invariant. In the case of the cubic fourfold $Y$ the Clemens--Griffiths component is either~$0$,
if $\CA_Y$ is equivalent to the derived category of a K3 surface, or $\CA_Y$ itself, in the opposite case.
So, Conjecture~\ref{4dconj} can be reformulated as stating that $Y$ is rational if and only if the Clemens--Griffiths
component of its derived category is zero.

Our goal is to give some evidence for Conjecture~\ref{4dconj}. More precisely, we will analyze
the category $\CA_Y$ for all cubic fourfolds that are known to be rational,
and show that in these cases $\CA_Y$ is equivalent to the derived category of some K3 surface.
Moreover, we will give examples of cubic fourfolds $Y$ for which $\CA_Y$ can not be equivalent
to the derived category of any K3 surface. So, we expect these cubic fourfolds to be nonrational.

Basically, there are three known series of rational cubic fourfolds:
\begin{enumerate}
\item Pfaffian cubic fourfolds~\cite{Tr1,Tr2};
\item cubic fourfolds $Y$ containing a plane $\BP \cong \PP^2$ and a $2$-dimensional algebraic cycle $T$
such that $T\cdot H^2 - T\cdot\BP$ is odd~\cite{Ha1,Ha2};
\item singular cubic fourfolds.
\end{enumerate}
We are dealing with these in Sections~\ref{secpf}, \ref{secplane} and~\ref{secsing} respectively
after introducing some necessary material in Section~\ref{secpre}. Moreover, in Section~\ref{secplane}
we investigate more general case of a cubic fourfold containing a plane without any additional conditions.
We show that in this case the category $\CA_Y$ is equivalent to the twisted derived category of a K3 surface $S$
(the twisting is given by a Brauer class of order 2) and in the Appendix we argue that this twisted derived
category is not equivalent to the derived category of any surface if $\Pic S \cong \ZZ$ and the Brauer class
is nontrivial, which is true for general cubic fourfold with a plane. So, we expect such cubic fourfolds
to be nonrational.

\medskip

{\bf Acknowledgement:} I am very grateful to A.~Bondal and D.~Orlov for useful discussions
and to L.~Katzarkov for inspiring and stimulating questions, and especially for attracting
my attention to the case of singular cubic fourfolds. I am also grateful to D.~Huybrechts
for clarifications on the twisted Chern character and to F.~Bogomolov and D.~Kaledin
for discussions on the Brauer group of K3 surfaces.

\section{Preliminaries}\label{secpre}

Let $\kk$ be an algebraically closed field of zero characteristic. All algebraic varieties will be
over $\kk$ and all additive categories will be $\kk$-linear.

By $\D^b(X)$ we denote the bounded derived category of coherent sheaves on an algebraic variety $X$.
This category is triangulated. For any morphism $f:X \to X'$ of algebraic varieties we denote
by $f_*:\D^b(X) \to \D^b(X')$ and by $f^*:\D^b(X') \to \D^b(X)$ the {\em derived}\/
push-forward and pull-back functors (in first case we need $f$ to be proper,
and in the second to have finite $\Tor$-dimension, these assumptions ensure the functors
to preserve both boundedness and coherence). Similarly, $\otimes$
stands for the {\em derived}\/ tensor product.

For a proper morphism of finite $\Tor$-dimension $f:X \to X'$ we will also use the right adjoint $f^!$
of the push-forward functor, which is given by the formula
$$
f^!(F) \cong f^*(F)\otimes\omega_{X/X'}[\dim X - \dim X'],
$$
where $\omega_{X/X'}$ is the relative canonical line bundle.

\subsection{Semiorthogonal decompositions}

Let $\CT$ be a triangulated category.

\begin{definition}[\cite{BK,BO}]
A {\sf semiorthogonal decomposition}\/ of a triangulated category $\CT$ is a sequence of
full triangulated subcategories $\CA_1,\dots,\CA_n$ in $\CT$ such that
$\Hom_{\CT}(\CA_i,\CA_j) = 0$ for $i > j$
and for every object $T \in \CT$ there exists a chain of morphisms
$0 = T_n \to T_{n-1} \to \dots \to T_1 \to T_0 = T$ such that
the cone of the morphism $T_k \to T_{k-1}$ is contained in $\CA_k$
for each $k=1,2,\dots,n$.
\end{definition}

We will write $\CT = \langle \CA_1,\CA_2,\dots,\CA_n \rangle$ for a semiorthogonal
decomposition of a triangulated category $\CT$ with components $\CA_1,\CA_2,\dots,\CA_n$.

An important property of a triangulated subcategory $\CA \subset \CT$ ensuring that
it can be extended to a semiorthogonal decomposition is admissibility.

\begin{definition}[\cite{BK,B}]
A full triangulated subcategory $\CA$ of a triangulated category $\CT$ is called
{\sf admissible}\/ if for the inclusion functor $i:\CA \to \CT$ there is
a right adjoint $i^!:\CT \to \CA$, and a left adjoint $i^*:\CT \to \CA$ functors.
\end{definition}

\begin{lemma}[\cite{BK,B}]\label{sos_sod}
$(i)$ If\/ $\CA_1,\dots,\CA_n$ is a semiorthogonal sequence
of admissible subcategories in a triangulated category $\CT$\/
{\rm(}i.e. $\Hom_{\CT}(\CA_i,\CA_j) = 0$ for $i > j${\rm)}\/ then
$$
\lan\CA_1,\dots,\CA_k,
{}^\perp\lan\CA_1,\dots,\CA_k\ran \cap \lan\CA_{k+1},\dots,\CA_n\ran{}^\perp,
\CA_{k+1},\dots,\CA_n\ran
$$
is a semiorthogonal decomposition.

$(ii)$ If\/ $\D^b(X) = \lan \CA_1,\CA_2,\dots,\CA_n \ran$ is a semiorthogonal decomposition
of the derived category of a smooth projective variety $X$ then each subcategory $\CA_i \subset \D^b(X)$
is admissible.
\end{lemma}

Actually the second part of the Lemma holds for any {\em saturated}\/ (see~\cite{BK}) triangulated category.

\begin{definition}[\cite{B}]
An object $F \in \CT$ is called {\em exceptional}\/ if $\Hom(F,F)=\kk$
and $\Ext^p(F,F)=0$ for all $p\ne 0$. A collection of exceptional
objects $(F_1,\dots,F_m)$ is called {\em exceptional}\/ if
$\Ext^p(F_l,F_k)=0$ for all $l > k$ and all $p\in\ZZ$.
\end{definition}

Assume that $\CT$ is {\em $\Hom$-finite}\/ (which means that for any $G,G' \in \CT$ the vector space $\oplus_{t\in\ZZ} \Hom(G,G'[t])$ is finite-dimensional).

\begin{lemma}[\cite{B}]\label{eo}
The subcategory $\lan F \ran$ of $\D^b(X)$ generated by an exceptional object $F$ is
admissible and is equivalent to the derived category of vector spaces $\D^b(\kk)$.
\end{lemma}
\begin{proof}
Consider the functor $\D^b(\kk) \to \D^b(X)$ defined by $V \mapsto V\otimes F$,
where $V$ is a complex of vector spaces. It is fully faithful since $F$ is exceptional,
hence the subcategory $\lan F \ran$ of $\D^b(X)$ is equivalent to~$\D^b(\kk)$. The adjoint functors
are given by $G \mapsto \RHom(G,F)^*$ and $G \mapsto \RHom(F,G)$ respectively.
\end{proof}

As a consequence of~\ref{sos_sod} and of~\ref{eo} one obtains the following

\begin{corollary}[\cite{BO}]
If $X$ is a smooth projective algebraic variety then any exceptional collection $F_1,\dots,F_m$ in $\D^b(X)$
induces a semiorthogonal decomposition
$$
\D^b(V) = \langle \CA , F_1, \dots, F_m \rangle
$$
where $\CA = \langle F_1, \dots, F_m \rangle^\perp = \{F \in \D^b(X)\ |\ \text{$\Ext^\bullet(F_k,F) = 0$ for all $1 \le k \le m$}\}$.
\end{corollary}

An example of such a semiorthogonal decomposition is given by~\eqref{dby}. Indeed, by~\eqref{cohoy}
the triple of line bundles $(\CO_Y, \CO_Y(1), \CO_Y(2))$ is an exceptional collection in $\D^b(Y)$, which gives
a semiorthogonal decomposition with the first component defined by~\eqref{defay}.

\subsection{Mutations}

If a triangulated category $\CT$ has a semiorthogonal decomposition then
usually it has quite a lot of them. More precisely, there are two groups
acting on the set of semiorthogonal decompositions --- the group of
autoequivalences of $\CT$, and a certain braid group. The action of the braid
group is given by the so-called mutations.

Roughly speaking, the mutated decomposition is obtained by dropping one of the components
of the decomposition and then extending the obtained semiorthogonal collection
by inserting new component at some other place as in Lemma~\ref{sos_sod}.
More precisely, the basic two operations are defined as follows.

\begin{lemma}[\cite{B}]\label{mutfun}
Assume that $\CA \subset \CT$ is an admissible subcategory, so that we have
two semiorthogonal decompositions $\CT = \lan \CA^\perp, \CA \ran$
and $\CT = \lan \CA, {}^\perp\CA \ran$.
Then there are functors $\LL_\CA,\RR_\CA: \CT \to \CT$
vanishing on~$\CA$ and inducing mutually inverse equivalences
${}^\perp\CA \to \CA^\perp$ and $\CA^\perp \to {}^\perp\CA$ respectively.
\end{lemma}

The functors $\LL_\CA$ and $\RR_\CA$ are known as the {\em left}\/ and the {\em right mutation functors}.

\begin{proof}
Let $i:\CA \to \CT$ be the embedding functor. For any $F \in \CT$ we define
$$
\LL_\CA(F) = \Cone (i i^! F \to F),
\qquad
\RR_\CA(F) = \Cone (F \to i i^* F)[-1].
$$
Note that the cones in this triangles are functorial due to the semiorthogonality.
All the properties are verified directly.
\end{proof}

\begin{remark}
If $\CA$ is generated by an exceptional object $E$ we can use explicit formulas for
the adjoint functors $i^!$, $i^*$ of the embedding functor $i:\CA \to \CT$. Thus we obtain
the following distinguished triangles
\begin{equation}\label{excmut}
\RHom(E,F)\otimes E \to F \to \LL_E(F),
\qquad
\RR_E(F) \to F \to \RHom(F,E)^*\otimes E.
\end{equation}
\end{remark}

It is easy to deduce from Lemma~\ref{mutfun} the following

\begin{corollary}[\cite{B}]
Assume that $\CT = \lan \CA_1,\CA_2,\dots,\CA_n \ran$ is a semiorthogonal decomposition with
all components being admissible. Then for each $1 \le k \le n-1$ there is a semiorthogonal decomposition
$$
\CT = \lan \CA_1,\dots,\CA_{k-1},
\LL_{\CA_k}(\CA_{k+1}),
\CA_k,\CA_{k+2},\dots,\CA_n \ran
$$
and for each $2 \le k \le n$ there is a semiorthogonal decomposition
$$
\CT = \lan \CA_1,\dots,\CA_{k-2},\CA_k,
\RR_{\CA_k}(\CA_{k-1}),
\CA_{k+1},\dots,\CA_n \ran
$$
\end{corollary}

There are two cases when the action of the mutation functors is particularly simple.

\begin{lemma}\label{perpmut}
Assume that $\CT = \lan \CA_1,\CA_2,\dots,\CA_n \ran$ is a semiorthogonal decomposition with
all components being admissible. Assume also that the components $\CA_k$ and $\CA_{k+1}$
are completely orthogonal, i.e. $\Hom(\CA_k,\CA_{k+1}) = 0$ as well as $\Hom(\CA_{k+1},\CA_k) = 0$. Then
$$
\LL_{\CA_k}(\CA_{k+1}) = \CA_{k+1}
\qquad\text{and}\qquad
\RR_{\CA_{k+1}}(\CA_k) = \CA_k,
$$
so that both the left mutation of $\CA_{k+1}$ through $\CA_k$ and the right mutation of $\CA_k$ through $\CA_{k+1}$
boil down to just a permutation and
$$
\CT = \lan \CA_1,\dots,\CA_{k-1},\CA_{k+1},\CA_k,\CA_{k+2},\dots,\CA_n \ran
$$
is the resulting semiorthogonal decomposition of $\CT$.
\end{lemma}

\begin{lemma}\label{longmut}
Let $X$ be a smooth projective algebraic variety and $\D^b(X) = \lan \CA,\CB \ran$ a semiorthogonal decomposition.
Then
$$
\LL_{\CA}(\CB) = \CB\otimes\omega_X
\qquad\text{and}\qquad
\RR_{\CB}(\CA) = \CA\otimes\omega_X^{-1}.
$$
\end{lemma}

An analogue of this Lemma holds for any triangulated category which has a Serre functor (see~\cite{BK}).
In this case tensoring by the canonical class should be replaced by the action of the Serre functor.

We will also need the following evident oservation.

\begin{lemma}\label{tensmut}
Let $\Phi$ be an autoequivalence of $\CT$. Then
$$
\Phi\circ\RR_\CA \cong \RR_{\Phi(\CA)}\circ\Phi,
\qquad
\Phi\circ\LL_\CA \cong \LL_{\Phi(\CA)}\circ\Phi.
$$
In particular, if $L$ is a line bundle on $X$
and $E$ is an exceptional object in $\D^b(X)$ then
$$
\TT_L\circ\RR_E \cong \RR_{E\otimes L}\circ\TT_L,
\qquad
\TT_L\circ\LL_E \cong \LL_{E\otimes L}\circ\TT_L,
$$
where $\TT_L:\D^b(X) \to \D^b(X)$ is the functor of tensor product by $L$.
\end{lemma}

\section{Pfaffian cubic fourfolds}\label{secpf}

Let $W$ be a vector space of dimension $6$. Consider $\PP(\Lambda^2W^*)$,
the space of skew-forms on $W$, and its closed subset
$$
\Pf(W) = \{ \omega \in \PP(\Lambda^2W^*)\ |\ \text{$\omega$ is degenerate} \}.
$$
It is well known that a skew form is degenerate if and only if its Pfaffian is zero,
so $\Pf(W)$ is a hypersurface in $\PP(\Lambda^2W^*)$ and its equation is given
by the Pfaffian in homogeneous coordinates of $\PP(\Lambda^2W^*)$. In particular,
$$
\deg \Pf(W) = \frac12\dim W = 3.
$$
We will say that $\Pf(W)$ is the {\sf Pfaffian cubic hypersurface}.

It is easy to check that $\Pf(W)$ is singular, its singularity is
the image of the Grassmanian $\Gr(2,W^*)$ under the Pl\"ucker embedding.
So,
$$
\codim_{\PP(\Lambda^2W^*)} \Big(\sing(\Pf(W))\Big) = \dim\PP(\Lambda^2W^*) - \dim \Gr(2,W^*) = 14 - 8 = 6.
$$
Let now $V$ be a 6-dimensional subspace of $\Lambda^2W^*$ such that $\PP(V) \not\subset \Pf(W^*)$.
Then $Y_V = \PP(V) \cap \Pf(W^*)$ is a cubic hypersurface in $\PP(V)$,
its equation being the restriction of the Pfaffian in coordinates of $\PP(\Lambda^2W^*)$ to~$\PP(V)$.
We will say that $Y_V$ is {\sf the Pfaffian cubic fourfold} associated with the subspace $V\subset \Lambda^2W^*$.
Note that thanks to the big codimension of $\sing (\Pf(W^*))$, the Pfaffian cubic fourfold
$Y_V$ is smooth for sufficiently generic $V$.

With each $V \subset \Lambda^2W^*$ we can also associate a global section $s_V$
of the vector bundle $V^*\otimes \CO_{\Gr(2,W)}(1)$ on the Grassmanian $\Gr(2,W)$
and its zero locus $X_V \subset \Gr(2,W)$.

The following was proved in~\cite{K3}.

\begin{theorem}
Let $Y_V$ be a smooth Pfaffian cubic fourfold. Then $X_V$ is a smooth K3 surface and there is
an equivalence of categories $\CA_{Y_V} \cong \D^b(X_V)$.
\end{theorem}

\begin{remark}
The same holds true even for singular Pfaffian cubic fourfolds as long as
$X_V$ is a  surface.
\end{remark}

\section{Cubic fourfolds with a plane}\label{secplane}

Let $Y \subset \PP(V)$ be a cubic fourfold containing a plane. In other words we assume
that there is a vector subspace $A \subset V$ of dimension 3 such that $\PP(A) \subset Y$.
Let $\sigma:\TY \to Y$ be the blowup of~$\PP(A)$. Let also $B = V/A$ be the quotient space.
Then the linear projection from $\PP(A)$ gives a regular map $\pi:\TY \to \PP(B) = \PP^2$.

\begin{lemma}
The map $\pi$ is a fibration in two-dimensional quadrics over $\PP(B)$ with the degeneration
curve of degree $6$. Let $D$ be the exceptional divisor of $\sigma$.
Let $i:D \to \TY$ be the embedding and $p:D \to \PP(A)$ the projection.
$$
\xymatrix{
& D \ar[r]^i \ar[dl]_p & \TY \ar[dl]_\sigma \ar[dr]^\pi \\ 
\PP(A) \ar[r] & Y && \PP(B) 
}
$$
Let $H$ and $h$ denote the pullbacks to $\TY$ of the classes of hyperplanes in $\PP(V)$ and $\PP(B)$ respectively.
Then we have linear equivalences
$$
D = H - h,\qquad
h = H - D,\qquad
K_{\TY} = -3H + D = -2H -h.
$$
\end{lemma}
\begin{proof}
Consider the blowup $\widetilde{\PP(V)}$ of $\PP(V)$ in $\PP(A)$. It is naturally projected to $\PP(B)$
and this projection identifies $\widetilde{\PP(V)}$ with $\PP_{\PP(B)}(E)$ for a vector bundle
$$
E = A\otimes\CO_{\PP(B)} \oplus \CO_{\PP(B)}(-1)
$$
on $\PP(B)$. The complete preimage of $Y$ under the map $\widetilde{\PP(V)} \to \PP(V)$ decomposes as the union of the proper preimage and the
exceptional divisor. It follows that $\TY = 3H' - D'$, where $H'$ stands for the pullback to $\widetilde{\PP(V)}$
of the class of a hyperplane in $\PP(V)$ and $D'$ stands for the exceptional divisor of $\widetilde{\PP(V)}$.
Note also that we have $h' = H' - D'$ on $\widetilde{\PP(V)}$, where $h'$ stands for the pullback to $\widetilde{\PP(V)}$
of the class of a hyperplane in $\PP(B)$. Thus $\TY = 2H' + h'$, so it follows that the fibers of $\TY$ over $\PP(B)$
are the quadrics in the fibers of $\widetilde{\PP(V)} = \PP_{\PP(B)}(E)$.
Moreover, it follows that the degeneration curve is the vanishing locus of the determinant of the morphism
$$
E \to E^*\otimes\CO_{\PP(B)}(1),
$$
given by the equation of $\TY$. So, the degeneration curve is the zero locus of a section of the line bundle $(\det E^*)^2\otimes\CO_{\PP(B)}(4) \cong \CO_{\PP(B)}(6)$.

Restricting the relation $h' = H' - D'$ to $\TY$ we deduce $h = H - D$. As for the canonical clas,
the first equality is evident since $\TY$ is a smooth blow-up, and replacing $D$ by $H - h$ we get the last equality.
\end{proof}

Now we will need the results of~\cite{K1} on the structure of the derived
category of coherent sheaves on a fibration in quadrics. According to loc.\ cit.
there is a semiorthogonal decomposition
\begin{equation}\label{pcso0}
\D^b(\TY) = \langle \Phi(\D^b(\PP(B),\CB_0)), \pi^*(\D^b(\PP(B)))\otimes\CO_\TY, \pi^*(\D^b(\PP(B)))\otimes\CO_\TY(H) \rangle,
\end{equation}
where $\CB_0$ is the sheaf of even parts of Clifford algebras corresponding to this fibration.
As a coherent sheaf on $\PP(B)$ it is given by the formula
$$
\CB_0 = \CO_{\PP(B)} \quad\oplus\quad \Lambda^2E\otimes\CO_{\PP(B)}(-1) \quad\oplus\quad \Lambda^4E\otimes\CO_{\PP(B)}(-2)
$$
and the multiplication law depends on the equation of $\TY$. Further, $\D^b(\PP(B),\CB_0)$ is the bounded
derived category of sheaves of coherent right $\CB_0$-modules on $\PP(B)$,
and the functor $\Phi:\D^b(\PP(B),\CB_0) \to \D^b(\TY)$
is defined as follows. Consider also the sheaf of odd parts of the corresponding Clifford algebras,
$$
\CB_1 = E \quad\oplus\quad \Lambda^3E\otimes\CO_{\PP(B)}(-1).
$$
Denote by $\alpha$ the embedding $\TY \to \PP_{\PP(B)}(E)$,
and by $q:\PP_{\PP(B)}(E) \to \PP(B)$ the projection.
Then there is a canonical map of left $q^*\CB_0$-modules $q^*\CB_0 \to q^*\CB_1(H')$,
which is injective and its cokernel is supported on~$\TY$. So, twisting by $\CO(-2H')$ for later convenience
we obtain an exact sequence
$$
0 \to q^*\CB_0(-2H') \to q^*\CB_1(-H') \to \alpha_*\CE \to 0,
$$
where $\CE$ is a sheaf of left $\pi^*\CB_0$-modules on $\TY$.
The functor $\Phi$ is just the kernel functor given by $\CE$, that is
$$
\Phi(F) = \pi^*F \otimes_{\pi^*\CB_0} \CE.
$$

At first glance, the category $\D^b(\PP(B),\CB_0)$ has nothing to do with K3 surfaces.
However, after a small consideration it is easy to see that it has.

From now on we assume that the fibers of $\TY$ over $\PP(B)$ are quadrics of rank $\ge 3$
(i.e.\ the don't degenerate into a union of two planes). This condition holds if $Y$
is sufficiently generic.
Consider the moduli space $\CM$ of lines contained in the fibers of $\TY$ over $\PP(B)$.
Since on a smooth two-dimensional quadric the moduli space of lines is a disjoint union of two $\PP^1$,
while on a singular quadric the moduli space of lines is one $\PP^1$, it follows that
$\CM$ is a $\PP^1$-fibration over the double cover $S$ of $\PP(B)$ ramified in the degeneration curve.
Note that $S$ is a K3 surface, and $\CM$ produces an element of order $2$ in the Brauer group of $S$.
Let $\CB$ be the corresponding sheaf of Azumaya algebras on $S$.
Denote by $f:S \to \PP(B)$ the double covering.

\begin{lemma}\label{ppb-s}
We have $f_*\CB \cong \CB_0$. In particular, $\D^b(S,\CB) \cong \D^b(\PP(B),\CB_0)$. The composition of the
equivalence with the functor $\Phi$ is given by
$$
F \mapsto \pi^*f_* F \otimes_{\pi^*\CB_0} \CE.
$$
\end{lemma}
\begin{proof}
The first claim is the classical property of the Clifford algebra. Indeed, $\CM$ is nothing but the isotropic
Grassmannian of half-dimensional linear subspaces in the quadrics, which always embeds into the projectivizations of
the half-spinor modules, and for a 2-dimensional quadric this embedding is an isomorphism.
So, (locally in the \'etale topology) $\CM$ is the projectivization of the half-spinor module.
On the other hand, the even part of the Clifford acts on the half-spinor modules and this action
identifies it with the product of their matrix algebras. Hence (locally in the \'etale topology)
the Clifford algebra is isomorphic to the pushforward of the endomorphism algebra of the half-spinor
module which is precisely the Azumaya algebra corresponding to its projectivization.

The second claim is evident (the equivalence is given
by the pushforward functor $f_*$).
\end{proof}

\begin{theorem}
There is an equivalence of categories $\CA_Y \cong \D^b(S,\CB)$.
\end{theorem}
\begin{proof}
Consider the semiorthogonal decomposition~\eqref{pcso0}. Replacing the first instance of $\D^b(\PP(B))$
by the exceptional collection $(\CO_{\PP(B)}(-1),\CO_{\PP(B)},\CO_{\PP(B)}(1))$ and the second
instance of $\D^b(\PP(B))$ by the exceptional collection $(\CO_{\PP(B)},\CO_{\PP(B)}(1),\CO_{\PP(B)}(2))$ we obtain a semiorthogonal decomposition
\begin{equation}\label{pcso1}
\D^b(\TY) = \langle \Phi(\D^b(\PP(B),\CB_0)), \CO_\TY(-h), \CO_\TY, \CO_\TY(h), \CO_\TY(H), \CO_\TY(h+H), \CO_\TY(2h+H) \rangle.
\end{equation}

On the other hand, since $\TY$ is the blowup of $Y$ in $\PP(A)$, we have by~\cite{Or} the following semiorthogonal decomposition
$$
\D^b(\TY) = \langle \sigma^*(\D^b(Y)), i_*p^*(\D^b(\PP(A))) \rangle.
$$
Replacing $\D^b(\PP(A))$ with the standard exceptional collection
$(\CO_{\PP(A)},\CO_{\PP(A)}(1),\CO_{\PP(A)}(2))$,
and $\D^b(Y)$ with its decomposition~\eqref{dby},
we obtain the following semiorthogonal decomposition
\begin{equation}\label{ssay}
\D^b(\TY) = \langle \sigma^*(\CA_Y), \CO_\TY, \CO_\TY(H), \CO_\TY(2H), i_*\CO_D, i_*\CO_D(H), i_*\CO_D(2H) \rangle.
\end{equation}
Now we are going to make a series of mutations, transforming decomposition~\eqref{pcso1} into~\eqref{ssay}.
This will give the required equivalence $\D^b(S,\CB) \cong \D^b(\PP(B),\CB_0) \cong \CA_Y$.

Now let us describe the series of mutations. We start with decomposition~\eqref{pcso1}.

\noindent{\bf Step 1.}
Right mutation of $\Phi(\D^b(\PP(B),\CB_0))$ through $\CO_{\TY}(-h)$.

After this mutation we obtain the following decomposition
\begin{equation}\label{pcso2}
\D^b(\TY) = \langle \CO_\TY(-h), \Phi'(\D^b(\PP(B),\CB_0)), \CO_\TY, \CO_\TY(h), \CO_\TY(H), \CO_\TY(h+H), \CO_\TY(2h+H) \rangle,
\end{equation}
where $\Phi' = \RR_{\CO_\TY(-h)}\circ\Phi:\D^b(\PP(B),\CB_0) \to \D^b(\TY)$.

\noindent{\bf Step 2.}
Right mutation of $\CO_{\TY}(-h)$ through the orthogonal subcategory ${}^\perp\langle \CO_{\TY}(-h) \rangle$.

Applying Lemma~\ref{longmut} and taking into account equality $K_\TY = -2H - h$, we obtain
%
\begin{equation}\label{pcso3}
\D^b(\TY) = \langle\Phi'(\D^b(\PP(B),\CB_0)), \CO_\TY, \CO_\TY(h), \CO_\TY(H), \CO_\TY(h+H), \CO_\TY(2h+H), \CO_\TY(2H) \rangle.
\end{equation}

\begin{lemma}
We have $\Ext^\bullet(\CO_{\TY}(2h+H), \CO_{\TY}(2H)) = 0$,
so that the pair $(\CO_{\TY}(2h+H), \CO_{\TY}(2H))$ is completely orthogonal.
\end{lemma}
\begin{proof}
We have
\begin{multline*}
\Ext^\bullet(\CO_{\TY}(2h+H), \CO_{\TY}(2H)) =
H^\bullet(\TY,\CO_{\TY}(H-2h)) = \\ =
H^\bullet(\PP(B),E^*\otimes\CO_{\PP(B)}(-2)) =
H^\bullet(\PP(B),A^*\otimes\CO_{\PP(B)}(-2) \oplus \CO_{\PP(B)}(-1))) = 0.
\end{multline*}
\end{proof}

By Lemma~\ref{perpmut} the transposition of the pair $(\CO_{\TY}(2h+H), \CO_{\TY}(2H))$ gives a semiorthogonal decomposition.

\noindent{\bf Step 3.}
Transpose the pair $(\CO_{\TY}(2h+H), \CO_{\TY}(2H))$.

After the transposition we obtain the following decomposition
\begin{equation}\label{pcso4}
\D^b(\TY) = \langle\Phi'(\D^b(\PP(B),\CB_0)), \CO_\TY, \CO_\TY(h), \CO_\TY(H), \CO_\TY(h+H), \CO_\TY(2H), \CO_\TY(2h+H) \rangle.
\end{equation}

\noindent{\bf Step 4.}
Left mutation of $\CO_{\TY}(2h+H)$ through the orthogonal subcategory $\langle \CO_{\TY}(2h+H) \rangle^\perp$.

Applying Lemma~\ref{longmut} and taking into account equality $K_\TY = -2H - h$, we obtain
%
%
\begin{equation}\label{pcso5}
\D^b(\TY) = \langle \CO_\TY(h-H), \Phi'(\D^b(\PP(B),\CB_0)), \CO_\TY, \CO_\TY(h), \CO_\TY(H), \CO_\TY(h+H), \CO_\TY(2H) \rangle.
\end{equation}

\noindent{\bf Step 5.}
Left mutation of $\Phi'(\D^b(\PP(B),\CB_0))$ through $\CO_{\TY}(h-H)$.

After this mutation we obtain the following decomposition
\begin{equation}\label{pcso6}
\D^b(\TY) = \langle \Phi''(\D^b(\PP(B),\CB_0)), \CO_\TY(h-H), \CO_\TY, \CO_\TY(h), \CO_\TY(H), \CO_\TY(h+H), \CO_\TY(2H) \rangle.
\end{equation}
where $\Phi'' = \LL_{\CO_\TY(h-H)}\circ\Phi':\D^b(\PP(B),\CB_0) \to \D^b(\TY)$.

\noindent{\bf Step 6.}
Simultaneous right mutation of $\CO_{\TY}(h-H)$ through $\CO_{\TY}$, of $\CO_{\TY}(h)$ through $\CO_{\TY}(H)$, and of $\CO_{\TY}(h+H)$ through $\CO_{\TY}(2H)$.

\begin{lemma}
We have $\RR_{\CO_{\TY}}(\CO_{\TY}(h-H)) \cong i_*\CO_D[-1]$,
$\RR_{\CO_{\TY}(H)}(\CO_{\TY}(h)) \cong i_*\CO_D(H)[-1]$, and also
$\RR_{\CO_{\TY}(2H)}(\CO_{\TY}(h+H)) \cong i_*\CO_D(2H)[-1]$.
\end{lemma}
\begin{proof}
Note that
\begin{multline*}
\Ext^\bullet(\CO_{\TY}(h-H), \CO_{\TY}) =
H^\bullet(\TY,\CO_{\TY}(H-h)) = \\ =
H^\bullet(\PP(B),E^*\otimes\CO_{\PP(B)}(-1)) =
H^\bullet(\PP(B),A^*\otimes\CO_{\PP(B)}(-1) \oplus \CO_{\PP(B)}).
\end{multline*}
It follows that
$$
\dim \Ext^p(\CO_{\TY}(h-H), \CO_{\TY}) =  \begin{cases}
1, & \text{for $p = 0$}\\
0, & \text{otherwise}
\end{cases}
$$
By~\eqref{excmut} we have the following distinguished triangle
$$
\RR_{\CO_{\TY}}(\CO_{\TY}(h-H)) \to \CO_{\TY}(h-H) \to \CO_\TY.
$$
Since $H - h = D$, the right map in the triangle is given by the equation of $D$,
so it follows that the first vertex is $i_*\CO_D[-1]$.
The same argument proves the second and the third claim.
\end{proof}

We conclude that the following semiorthogonal decomposition is obtained:
\begin{equation}\label{pcso7}
\D^b(\TY) = \langle \Phi''(\D^b(\PP(B),\CB_0)), \CO_\TY, i_*\CO_D, \CO_\TY(H), i_*\CO_D(H), \CO_\TY(2H), i_*\CO_D(2H) \rangle.
\end{equation}

\begin{lemma}
We have $\Ext^\bullet(i_*p^*F, \pi^*G) = 0$,
so that $i_*\CO_D(H)$ is completely orthogonal to $\CO_\TY(2H)$ and
so that $i_*\CO_D$ is completely orthogonal to $\langle \CO_{\TY}(H),\CO_{\TY}(2H) \rangle$.
\end{lemma}
\begin{proof}
Since $i$ is a divisorial embedding, we have $\omega_{D/\TY} \cong \CO_D(D)$, so $i^!(-) \cong i^*(-)\otimes\CO_D(D)[-1]$.
Therefore we have
\begin{multline*}
\Ext^\bullet(i_*p^*F, \pi^*G) =
\Ext^\bullet(p^*F, i^!\pi^*G) =
\Ext^\bullet(p^*F, i^*\pi^*G(D)[-1]) = \\ =
\Ext^\bullet(p^*F, p^*j^*G(D)[-1]) =
\Ext^\bullet(F, p_*(p^*j^*G(D))[-1]) =
\Ext^\bullet(F, j^*G\otimes p_*\CO_D(D)[-1]) = 0,
\end{multline*}
the last equality is satisfied because $p:D \to \PP(A)$ is a $\PP^1$-fibration
and $\CO_D(D) = \CO_D(H-h)$ restricts as $\CO(-1)$ to all its fibers.
\end{proof}

By Lemma~\ref{perpmut} the transposition of $i_*\CO_D(H)$ to the right of $\CO_\TY(2H)$,
and of $i_*\CO_D$ to the right of the subcategory $\langle \CO_{\TY}(H),\CO_{\TY}(2H) \rangle$
in~\eqref{pcso7} gives a semiorthogonal decomposition.

\noindent{\bf Step 7.}
Transpose $i_*\CO_D(H)$ to the right of $\CO_\TY(2H)$,
and $i_*\CO_D$ to the right of $\langle \CO_{\TY}(H),\CO_{\TY}(2H) \rangle$.

After the transposition we obtain the following decomposition
%
%
%
\begin{equation}\label{pcso8}
\D^b(\TY) = \langle \Phi''(\D^b(\PP(B),\CB_0)), \CO_\TY, \CO_\TY(H), \CO_\TY(2H), i_*\CO_D, i_*\CO_D(H), i_*\CO_D(2H) \rangle.
\end{equation}

Now we are done.
Comparing~\eqref{pcso8} with~\eqref{ssay}, we see that $\Phi''(\D^b(\PP(B),\CB_0)) = \sigma^*(\CA_Y)$,
hence the functor
$$
\begin{array}{c}
\sigma_*\circ\Phi'':\D^b(\PP(B),\CB_0) \to \CA_Y,\\
F \mapsto \left\{ \Hom(\CO_\TY(h-H),\Phi(F)) \otimes J_{\PP(A)} \to \Phi(F) \to \Hom(\Phi(F),\CO_\TY(-h))^*\otimes\CO_Y(-H) \right\}.
\end{array}
$$
is an equivalence of categories. Combining this with the equivalence of Lemma~\ref{ppb-s}
we obtain an equivalence $\D^b(S,\CB) \cong \CA_Y$.
\end{proof}

The category $\D^b(S,\CB)$ can be considered as a {\em twisted}\/ derived category
of the K3 surface $S$, the twisting being given by the class of $\CB$ in the Brauer group of $S$.
However, sometimes the twisting turns out to be trivial.

For a 2-dimensional cycle $T$ on a cubic fourfold $Y$ containing a plane $\BP$
consider the intersection index
\begin{equation}\label{dt}
\delta(T) = T\cdot H \cdot H - T \cdot \BP.
\end{equation}
Note that $\delta(\BP) = -2$, and that $\delta(H^2) = 2$.
So, if the group of 2-cycles on $Y$ modulo numerical equivalence
is generated by $\BP$ and $H^2$ then $\delta$ takes only even values.

\begin{proposition}\label{btriv}
The sheaf of Azumaya algebras $\CB$ on $S$ splits
if and only if there exists a $2$-dimensional cycle $T$ on $Y$
such that $\delta(T)$ is odd.
\end{proposition}
\begin{proof}
By definition the $\PP^1$-fibration over $S$ corresponding to the sheaf of Azumaya algebras $\CB$
is given by the moduli space $\CM$ of lines in the fibers of $\TY$ over $\PP(B)$, hence
$\CB$ splits if and only if the map $\CM \to S$ has a rational multisection of odd degree.
Note that $\delta(T)$ equals the intersection index of the proper preimage of $T$
in $\TY$ with the fiber of $\TY$ over $\PP(B)$. Hence, the set of lines in the fibers of~$\TY$
that intersect $T$ gives a multisection of $\CM \to S$ of degree $\delta(T)$.

To prove the converse, consider the component
$\CM^{(2)}$ of $\CM\times_{\PP(B)}\CM$ lying over the graph of the involution of $S$
over $\PP(B)$ in $S\times_{\PP(B)} S$. The points of $\CM^{(2)}$ correspond to pairs
of lines in the fibers of~$\TY$ over $\PP(B)$ lying in different families.
Associating with such pair of lines the point of their intersection,
we obtain a rational map $\CM^{(2)} \to \TY \to Y$. If $Z \subset \CM$ is a rational
multisection of $\CM \to S$ of odd degree~$d$, then we take $T$ to be (the closure of)
the image of $Z^{(2)}$ (defined analogously to $\CM^{(2)}$) in $Y$.
Then it is easy to see that $\delta(T) = d^2$ is odd.
\end{proof}

Note that the cubic fourfolds containing a plane $\PP(A)$ and a 2-dimensional cycle $T$
such that $\delta(T)$ is odd are rational by results of Hassett~\cite{Ha1,Ha2}.
So, by Proposition~\ref{btriv} for this series of rational cubic fourfolds
the category $\CA_Y$ is equivalent to the derived category of a K3 surface.

On the other hand, if a cubic fourfold $Y$ contains a plane
but $\delta(T)$ is even for any 2-cycle $T$ on $Y$, so that
the sheaf of Azumaya algebras $\CB$ doesn't split, it is still
possible that $\D^b(S,\CB) \cong \D^b(S')$ for some other
(or even for the same) K3 surface $S'$. For this, however,
the Picard group of $S$ should be sufficiently big, at least
if $\Pic S \cong \ZZ$ this is impossible.

\begin{proposition}\label{dbsb}
Let $Y$ be a cubic fourfold containing a plane $\BP$.
If the group of codimension $2$ algebraic cycles on $Y$
modulo numerical equivalence is generated by $\BP$ and $H^2$
then $\CA_Y \not\cong \D^b(S')$ for any surface~$S'$.
%
\end{proposition}

The proof will be given in the Appendix.
Actually, we will show that $K_0(S,\CB) \not\cong K_0(S')$
for any surface $S'$ as lattices with a bilinear form (the Euler form $\chi$),
where $S$ is the K3 surface corresponding to $Y$ and $\CB$ is the induced sheaf of Azumaya algebras on it.
More precisely, we will show that $K_0(S,\CB)$ doesn't contain a pair of vectors
$(v_1,v_2)$ such that $\chi_{(S,\CB)}(v_1,v_2) = 1$, $\chi_{(S,\CB)}(v_2,v_2) = 0$
(while in $K_0(S')$ one can take $v_1 = [\CO_{S'}]$, $v_2 = [\CO_p]$,
where $\CO_p$ is a structure sheaf of a point).


\section{Singular cubic fourfolds}\label{secsing}

Let $Y \subset \PP(V) = \PP^5$ be a singular cubic fourfold. Let $P$ be its singular point.
Let $\sigma:\TY \to Y$ be the blowup of $P$. The linear projection from $P$
gives a regular map $\pi:\TY \to \PP^4$.

\begin{lemma}
The map $\pi$ is the blowup of a K3 surface $S \subset \PP^4$
which is an intersection of a quadric and a cubic hypersurfaces.
Let $D$ be the exceptional divisor of $\pi$, and $Q$ the exceptional divisor of $\sigma$.
Let $i:D \to \TY$, $j:S \to \PP^4$, and $\alpha:Q \to \TY$ be the embeddings, and $p:D \to S$ the projection.
$$
\xymatrix{
& Q \ar[r]^\alpha \ar[dl] & \TY \ar[dl]_\sigma \ar[dr]^\pi & D \ar[l]_i \ar[dr]^p \\
P \ar[r] & Y && \PP^4 & S \ar[l]_j
}
$$
Then the map $\pi\circ\alpha:Q \to \PP^4$ identifies $Q$ with the quadric passing through $S$.

Moreover, let $H$ and $h$ denote the pullbacks to $\TY$ of the classes of hyperplanes in $\PP(V)$ and $\PP^4$ respectively.
Then we have the following relations in $\Pic\TY$:
$$
Q = 2h - D,\qquad
H = 3h - D,\qquad
h = H - Q,\quad
D = 2H - 3Q,\quad
K_{\TY} = -5h + D = -3H + 2Q.
$$
\end{lemma}
\begin{proof}
Let $z_0,\dots,z_5$ be coordinates in $V$ such that $P = (1:0:0:0:0:0)$.
Then $(z_1:z_2:z_3:z_4:z_5)$ are the homogeneous coordinates in $\PP^4$
and the rational map $\pi\circ\sigma^{-1}$ takes $(z_0:z_1:z_2:z_3:z_4:z_5)$ to $(z_1:z_2:z_3:z_4:z_5)$.
Since the point $P$ is singular for $Y$, the equation of $Y$ is given by
$$
z_0F_2(z_1,z_2,z_3,z_4,z_5) + F_3(z_1,z_2,z_3,z_4,z_5) = 0,
$$
where $F_2$ and $F_3$ are homogeneous forms of degree $2$ and $3$ respectively.
It follows that $\pi$ is the blowup of the surface
$$
S = \{ F_2(z_1,z_2,z_3,z_4,z_5) = F_3(z_1,z_2,z_3,z_4,z_5) = 0 \} \subset \PP^4,
$$
and that the rational map $\sigma\circ\pi^{-1}$ is given by the formula
$$
(z_1:z_2:z_3:z_4:z_5) \mapsto (-F_3(z_1,z_2,z_3,z_4,z_5):z_1F_2(z_1,z_2,z_3,z_4,z_5):\ \dots\ :z_5F_2(z_1,z_2,z_3,z_4,z_5)).
$$
It follows that the proper preimage of the quadric $\{F_2(z_1,z_2,z_3,z_4,z_5) = 0\} \subset \PP^4$
in $\TY$ is contracted by $\sigma$, whence $Q = 2h - D$. It also follows that $H = 3h - D$.
Solving these with respect to $h$ and $D$ we get the other two relations.
Finally, since $\TY$ is the blow-up of $\PP^4$ in a surfaces $S$ we have $K_\TY = -5h + D$.
Substituting $h = H - Q$, $D = 2H - 3Q$ we deduce the last equality.
\end{proof}

\begin{theorem}\label{dbsc}
The category $D^b(S)$ is a crepant categorical resolution of $\CA_Y$.
In other words, there exists a pair of functors
$$
\rho_*:\D^b(S) \to \CA_Y,
\qquad
\rho^*:\CA_Y^\perf \to D^b(S),
$$
where $\CA_Y^\perf = \CA_Y \cap \D^\perf(Y)$, such that $\rho^*$ is both left and right adjoint to $\rho_*$
and $\rho_*\circ\rho^* \cong \id$.
\end{theorem}

The notion of a crepant categorical resolution of singularities was introduced in~\cite{K2}.
To prove the Theorem we start with considering a crepant categorical resolution $\TD$
of~$\D^b(Y)$. Following~\cite{K2}, to construct such $\TD$ one starts with a dual Lefschetz
decomposition (with respect to the conormal bundle) of the derived category of the exceptional
divisor of a usual resolution. We take the resolution $\sigma:\TY \to Y$. Then the exceptional
divisor is a three-dimensional quadric $Q$ and the conormal bundle is isomorphic to $\CO_Q(-Q) \cong \CO_Q(h-H) \cong \CO_Q(h)$.
We choose the dual Lefschetz decomposition
$$
\D^b(Q) = \langle \CB_2\otimes\CO_Q(-2h),\CB_1\otimes\CO_Q(-h),\CB_0 \rangle
\qquad
\text{with $\CB_2 = \CB_1 = \langle \CO_Q \rangle$ and $\CB_0 = \langle \CO_Q, \CS_Q \rangle$},
$$
where $\CS_Q$ is the spinor bundle on $Q$. Then by~\cite{K2} the triangulated category
$$
\TD = {}^\perp\langle \alpha_*(\CB_2\otimes\CO_Q(-2h)),\alpha_*(\CB_1\otimes\CO_Q(-h)) \rangle =
{}^\perp\langle \alpha_*\CO_Q(-2h),\alpha_*\CO_Q(-h) \rangle
$$
is a crepant categorical resolution of $\D^b(Y)$ and there is a semiorthogonal decomposition
$$
\D^b(\TY) = \langle \alpha_*\CO_Q(-2h),\alpha_*\CO_Q(-h), \TD \rangle.
$$
Further, we consider the semiorthogonal decomposition of $\TD$, induced
by the decomposition~\eqref{dby} of $\D^b(Y)$:
$$
\TD = \langle \TCA_Y, \CO_\TY, \CO_\TY(H), \CO_\TY(2H) \rangle.
$$
One can easily show that $\TCA_Y$ is a crepant categorical resolution of $\CA_Y$ (see Lemma~\ref{acr} below).
So, to prove the Theorem, it suffices to check that $\TCA_Y \cong \D^b(S)$.

Now we describe the way we check this.
Substituting the above decomposition of $\TD$ into the above decomposition of $\D^b(\TY)$ we obtain
\begin{equation}\label{tcay}
\D^b(\TY) = \langle \alpha_*\CO_Q(-2h),\alpha_*\CO_Q(-h), \TCA_Y, \CO_\TY, \CO_\TY(H), \CO_\TY(2H) \rangle.
\end{equation}
%
%
On the other hand, since $\pi:\TY \to \PP^4$ is the blow-up of $S$ we have by~\cite{Or} the following semiorthogonal decomposition
$$
\D^b(\TY) = \langle \Phi(\D^b(S)), \pi^*(\D^b(\PP^4)) \rangle,
$$
where the functor $\Phi:\D^b(S) \to \D^b(\TY)$ is given by $F \mapsto i_*p^*F(D)$.
Using one of the standard exceptional collections
$\D^b(\PP^4) = \langle \CO_{\PP^4}(-3), \CO_{\PP^4}(-2), \CO_{\PP^4}(-1), \CO_{\PP^4}, \CO_{\PP^4}(1) \rangle$
we obtain a semiorthogonal decomposition
\begin{equation}\label{scso1}
\D^b(\TY) = \langle \Phi(\D^b(S)), \CO_{\TY}(-3h), \CO_{\TY}(-2h), \CO_{\TY}(-h), \CO_{\TY}, \CO_{\TY}(h)  \rangle.
\end{equation}

Now we are going to make a series of mutations, transforming decomposition~\eqref{scso1} into~\eqref{tcay}.
This will give the required equivalence $\D^b(S) \cong\TCA_Y$.

Now let us describe the series of mutations.

\noindent{\bf Step 1.}
Left mutation of $\CO_{\TY}(-3h)$, $\CO_{\TY}(-2h)$, and $\CO_{\TY}(-h)$ through $\Phi(\D^b(S))$.

\begin{lemma}\label{lphi}
For any $F \in \D^b(\PP^4)$ we have $\LL_{\Phi(\D^b(S))}(\pi^* F) = \pi^* F(D)$.
\end{lemma}
\begin{proof}
Recall that by definition we have
\begin{equation}\label{llphi}
\LL_{\Phi(\D^b(S))}(G) = \Cone(\Phi(\Phi^!(G)) \to G)
\end{equation}
for any $G \in \D^b(\TY)$. But
$$
\Phi^!(\pi^*F) \cong p_*i^!\pi^*F(-D) \cong p_*i^*\pi^*F[-1] \cong p_*p^*j^*F[-1] \cong j^*F[-1],
$$
so
$$
\Phi(\Phi^!(\pi^*F)) \cong i_*p^*j^*F(D)[-1] \cong i_*i^*\pi^*F(D)[-1]
$$
and it is clear that the triangle~\eqref{llphi} boils down to
$$
i_*i^*\pi^*F(D)[-1] \to \pi^*F \to \pi^*F(D)
$$
obtained by tensoring exact sequence $0 \to \CO_\TY \to \CO_\TY(D) \to i_*\CO_D(D) \to 0$
with $\pi^*F$ and rotating to the left.
\end{proof}

It follows that after this mutation we obtain the following decomposition
\begin{equation}\label{scso2}
\D^b(\TY) = \langle \CO_{\TY}(-3h+D), \CO_{\TY}(-2h+D), \CO_{\TY}(-h+D), \Phi(\D^b(S)), \CO_{\TY}, \CO_{\TY}(h)  \rangle.
\end{equation}

\noindent{\bf Step 2.}
Right mutation of $\Phi(\D^b(S))$ through $\CO_\TY$ and $\CO_\TY(h)$.

After this mutation we obtain the following decomposition
\begin{equation}\label{scso3}
\D^b(\TY) = \langle \CO_{\TY}(-3h+D), \CO_{\TY}(-2h+D), \CO_{\TY}(-h+D), \CO_{\TY}, \CO_{\TY}(h), \Phi'(\D^b(S))  \rangle,
\end{equation}
where $\Phi' = \RR_{\CO_\TY(h)}\circ\RR_{\CO_\TY}\circ\Phi: \D^b(S) \to \D^b(\TY)$.

\begin{lemma}
We have $\Ext^\bullet(\CO_{\TY}(-h+D), \CO_{\TY}) = 0$,
so that the pair $(\CO_{\TY}(-h+D), \CO_{\TY})$ is completely orthogonal.
\end{lemma}
\begin{proof}
We have $\Ext^\bullet(\CO_{\TY}(-h+D), \CO_{\TY}) = H^\bullet(\TY,\CO_{\TY}(h-D)) = H^\bullet(\PP^4,J_S(1))$,
where $J_S$ is the sheaf of ideals of $S$. On the other hand, since $S$ is the intersection of a quadric
and a cubic, we have an exact sequence
$$
0 \to \CO_{\PP^4}(-5) \to \CO_{\PP^4}(-3) \oplus \CO_{\PP^4}(-2) \to J_S \to 0.
$$
Twisting by $\CO_{\PP^4}(1)$ we deduce the required vanishing.
\end{proof}

By Lemma~\ref{perpmut} the transposition of the pair $(\CO_{\TY}(-h+D), \CO_{\TY})$ gives a semiorthogonal decomposition.

\noindent{\bf Step 3.}
Transpose the pair $(\CO_{\TY}(-h+D),\CO_{\TY})$.

After the transposition we obtain the following decomposition
\begin{equation}\label{scso4}
\D^b(\TY) = \langle \CO_{\TY}(-3h+D), \CO_{\TY}(-2h+D), \CO_{\TY}, \CO_{\TY}(-h+D), \CO_{\TY}(h), \Phi'(\D^b(S))  \rangle.
\end{equation}

\noindent{\bf Step 4.}
Simultaneous right mutation of $\CO_{\TY}(-2h+D)$ through $\CO_{\TY}$ and of $\CO_{\TY}(-h+D)$ through $\CO_{\TY}(h)$.

\begin{lemma}
We have $\RR_{\CO_{\TY}}(\CO_{\TY}(-2h+D)) \cong \alpha_*\CO_Q[-1]$,  $\RR_{\CO_{\TY}(h)}(\CO_{\TY}(-h+D)) \cong \alpha_*\CO_Q(h)[-1]$.
\end{lemma}
\begin{proof}
Note that $\Ext^\bullet(\CO_{\TY}(-2h+D), \CO_{\TY}) = H^\bullet(\TY,\CO_{\TY}(2h-D)) = H^\bullet(\PP^4,J_S(2))$.
Using the above resolution of $J_S$ we deduce that
$$
\dim \Ext^p(\CO_{\TY}(-2h+D), \CO_{\TY}) =  \begin{cases}
1, & \text{for $p = 0$}\\
0, & \text{otherwise}
\end{cases}
$$
It follows that we have a distinguished triangle
$$
\RR_{\CO_{\TY}}(\CO_{\TY}(-2h+D)) \to \CO_{\TY}(-2h+D) \to \CO_\TY.
$$
Since $2h - D = Q$, the right map in the triangle is given by $Q$, so it follows that the first vertex is $\alpha_*\CO_Q[-1]$.
The same argument proves the second claim.
\end{proof}

We conclude that the following semiorthogonal decomposition is obtained:

\begin{equation}\label{scso5}
\D^b(\TY) = \langle \CO_{\TY}(-3h+D), \CO_{\TY}, \alpha_*\CO_Q, \CO_{\TY}(h), \alpha_*\CO_Q(h), \Phi'(\D^b(S))  \rangle.
\end{equation}

\noindent{\bf Step 5.}
Left mutation of $\CO_{\TY}(h)$ through $\alpha_*\CO_Q$.

\begin{lemma}
We have $\LL_{\alpha_*\CO_Q}(\CO_{\TY}(h)) \cong \CO_{\TY}(3h-D)$.
\end{lemma}
\begin{proof}
Note that
$$
\Ext^\bullet(\alpha_*\CO_Q,\CO_{\TY}(h)) =
\Ext^\bullet(\CO_Q,\alpha^!\CO_{\TY}(h)) =
\Ext^\bullet(\CO_Q,\CO_Q(h+Q)[-1]) =
\Ext^\bullet(\CO_Q,\CO_Q(3h-D)[-1]).
$$
But the divisor $Q$ is contracted by $\sigma$, and $3h - D = H$ is a pullback by $\sigma$,
hence $\CO_Q(3h-D) \cong \CO_Q$, so we conclude that
$$
\dim \Ext^p(\alpha_*\CO_Q,\CO_{\TY}(h)) =  \begin{cases}
1, & \text{for $p = 1$}\\
0, & \text{otherwise}
\end{cases}
$$
It follows that we have a distinguished triangle
$$
\alpha_*\CO_Q[-1] \to \CO_{\TY}(h) \to \LL_{\alpha_*\CO_Q}(\CO_{\TY}(h)).
$$
Comparing it with the rotation of the exact sequence
$$
0 \to \CO_\TY(h) \to \CO_\TY(3h-D) \to \alpha_*\CO_Q \to 0
$$
we deduce the required isomorphism.
\end{proof}

As a result of this step we obtain the following semiorthogonal decomposition:

\begin{equation}\label{scso6}
\D^b(\TY) = \langle \CO_{\TY}(-3h+D), \CO_{\TY}, \CO_{\TY}(3h - D),  \alpha_*\CO_Q,\alpha_*\CO_Q(h), \Phi'(\D^b(S))  \rangle.
\end{equation}

\noindent{\bf Step 6.}
Left mutation of the subcategory $\langle \alpha_*\CO_Q,\alpha_*\CO_Q(h), \Phi'(\D^b(S)) \rangle$ through its orthogonal subcategory
$\langle \alpha_*\CO_Q,\alpha_*\CO_Q(h), \Phi'(\D^b(S))  \rangle^\perp =
\langle \CO_{\TY}(-3h+D), \CO_{\TY}, \CO_{\TY}(3h - D) \rangle$
and a twist by $\CO_\TY(H)$.

Applying Lemma~\ref{longmut} and taking into account equalities $K_\TY = D - 5h$ and $3h - D = H$, we obtain
%
%
\begin{equation}\label{scso7}
\D^b(\TY) = \langle \alpha_*\CO_Q(-2h),\alpha_*\CO_Q(-h), \Phi''(\D^b(S)), \CO_{\TY}, \CO_{\TY}(H), \CO_{\TY}(2H) \rangle,
\end{equation}
where
$$
\Phi'' = \TT_{\CO_\TY(3h-D)} \circ \TT_{\CO_\TY(-5h+D)} \circ \Phi' = \TT_{\CO_\TY(-2h)} \circ \Phi'.
$$

Comparing~\eqref{scso7} with~\eqref{tcay} we obtain the following

\begin{corollary}\label{phipp}
The functor $\Phi'' = \RR_{\CO_\TY(-h)} \circ \RR_{\CO_\TY(-2h)} \circ \TT_{\CO_\TY(D-2h)} \circ i_* \circ p^* : \D^b(S) \to \D^b(\TY)$
induces an equivalence of categories $\D^b(S) \cong \TCA_Y$.
\end{corollary}

Now we can finish the proof of Theorem~\ref{dbsc} by the following

\begin{lemma}\label{acr}
The category $\TCA_Y$ is a crepant categorical resolution of the category~$\CA_Y$.
\end{lemma}
\begin{proof}
Recall that by~\cite{K2} the category $\TD$ is a crepant categorical resolution of $\D^b(Y)$ via the functors
$\sigma_*:\TD \to \D^b(Y)$ and $\sigma^*:\D^\perf(Y) \to \TD$. So, to prove the lemma we only have to check
that $\sigma_*(\TCA_Y) \subset \CA_Y$ and $\sigma^*(\CA_Y^\perf)\subset \TCA_Y$.
But this is straightforward --- if $F \in \TCA_Y$ then
$$
\Hom(\CO_Y(t),\sigma_*(F)) =
\Hom(\sigma^*(\CO_Y(t)),F) =
\Hom(\CO_\TY(tH)),F) = 0
$$
for $t = 0,1,2$ by adjunction between $\sigma_*$ and $\sigma^*$ and the definition of $\TCA_Y$.
Similarly, if $G \in \CA_Y^\perf$ then
$$
\Hom(\CO_\TY(tH)),\sigma^*(G)) =
\Hom(\sigma^*(\CO_Y(t)),\sigma^*(G)) =
\Hom(\CO_Y(t),\sigma_*\sigma^*(G)) \cong
\Hom(\CO_Y(t),G) = 0
$$
again by adjunction and by the fact that $\sigma_*\circ\sigma^* \cong \id$ on $\D^\perf(Y)$.
\end{proof}

\begin{proofof}{Proof of Theorem~\ref{dbsc}}
By Corollary~\ref{phipp} we have $\D^b(S) \cong \TCA_Y$ and by Lemma~\ref{acr} the category
$\TCA_Y$ is a crepant categorical resolution of~$\CA_Y$.
\end{proofof}

\begin{remark}
The resulting functors between $\D^b(S)$ and $\CA_Y$ take form
$$
\begin{array}{c}
\rho_* = \sigma_* \circ \RR_{\CO_\TY(-h)} \circ \RR_{\CO_\TY(-2h)} \circ \TT_{\CO_\TY(D-2h)} \circ i_* \circ p^* : \D^b(S) \to \CA_Y,\\
F \mapsto \{ \sigma_*i_*(p^*F(D-2h)) \to \Hom^\bullet(F,\CO_S[-1])^*\otimes\CO_\TY(-2H) \to \Hom(F,\Omega_{\PP^4}(h)_{|S}[-1])^*\otimes\CO_\TY(-H) \}\\[1ex]
\rho^* = p_*\circ i^*\circ \TT_{\CO_\TY(-3h)[2]}\circ \LL_{\CO_\TY(2h-D)}\circ \LL_{\CO_\TY(4h-D)} \circ \sigma^*: \CA_Y^\perf \to \D^b(S),\\
G \mapsto p_*(i^*\sigma^*G(-3h))[2].
\end{array}
$$
\end{remark}
It will be interesting to describe the subcategory of $\D^b(S)$
which is mapped to zero by $\rho_*:\D^b(S) \to \CA_Y$.


\section{Appendix. The Grothendieck group of a twisted K3 surface\\associated with a cubic fourfold containing a plane}


For the computation of the Grothendieck group we will use the notion of a twisted Chern character
introduced by Huybrechts and Stellari in~\cite{HS1}, see also~\cite{HS2}.

Let $S$ be a polarized K3 surface of degree $2$ with $\Pic S = \ZZ h$ (and $h^2 = 2$),
and $\CB$ a nonsplit Azumaya algebra of rank $4$ on $S$.
Let $\alpha \in H^2(S,\CO_S^*)$ be the class of $\CB$ in the Brauer group $\Br(S) = H^2(S,\CO_S^*)$.
It follows from the exponential sequence
$$
\xymatrix@1{0 \ar[rr] && \ZZ  \ar[rr] && \CO_S \ar[rr]^{\exp(2\pi i(-))} &&\ \CO_S^* \ar[rr] && 0}
$$
and from $H^3(S,\ZZ) = 0$ that there exists $\beta \in H^2(S,\CO_S)$ such that $\alpha = \exp(2\pi i \beta)$.
Since the order of $\alpha$ is $2$, it follows that $2\beta$ is the image of some integer class $B_0 \in H^2(S,\ZZ)$,
hence $\beta$ is the image of $B = \frac12 B_0 \in H^2(S,\frac12\ZZ) \subset H^2(S,\QQ)$. Let us fix such $B$.

Certainly, the class $B \in H^2(S,\frac12\ZZ)$ such that $\alpha = \exp(2\pi i B)$ is not unique.
It is defined up to addition of an element in $H^2(S,\ZZ)$ (ambiguity in a choice of $\beta$)
and of an element in $H^{1,1}(S,\frac12\ZZ)$ (ambiguity in the lifting of $\beta$ to $B$).
However, the following invariant does not depend on the choices.

Let $\{t\} := t -\lfloor t\rfloor$ denote the fractional part of $t \in \QQ$.

\begin{lemma}\label{bhbb}
The fractional part $\{ Bh \}$ of the product $Bh$ does not depend on the choice of $B$.
Moreover, if $\{ Bh \} = \frac12$ then $\{ B^2 \}$ does not depend on the choice of $B$.
\end{lemma}
\begin{proof}
Take any $u \in H^2(S,\ZZ)$. Then
$$
(B+u)h = Bh + uh,\qquad
(B+u)^2 = B^2 + 2Bu + u^2.
$$
It is clear that $uh$, $(2B)u$ and $u^2$ are integral, so $\{Bh\}$ and $\{B^2\}$ do not change.
Further
$$
(B+\frac12 h)h = Bh + \frac12 h^2 = Bh + 1,\qquad
(B+\frac12 h)^2 = B^2 + Bh + \frac14 h^2 = B^2 + Bh + \frac12.
$$
We see that $\{Bh\}$ doesn't change and that if $\{Bh\} = \frac12$ then $\{B^2\}$ doesn't change as well.
\end{proof}

One can compute $\{Bh\}$ from the geometry of the $\PP^1$-bundle associated with the corresponding Brauer class $\alpha = \exp(2\pi i B)$.
It is an interesting question, how to compute $\{ B^2 \}$ in a similar fashion.

\begin{lemma}\label{bh}
Let $\CM \to S$ be a $\PP^1$-fibration
given by the Brauer class $\alpha = \exp(2\pi i B)$ and $C \subset S$, a smooth curve
in the linear system $|h|$ on $S$. Then there exists a rank $2$ vector bundle $E$ on $C$
such that $\CM\times_S C \cong \PP_C(E)$. Moreover
$$
\{ Bh \} = \left\{ \frac12\deg\det E \right\},
$$
for any such bundle $E$.
\end{lemma}
\begin{proof}
A vector bundle $E$ exists since the Brauer group of a smooth curve is trivial.
The equality of fractional parts can be deduced as follows. Note that the class
$\alpha \in H^2(S,\CO_S^*)$ is of order $2$, hence it comes from a class $\alpha_0 \in H^2(S,\mu_2)$,
where $\mu_2$ stands for the group of square roots of unity.
Note also that we have a canonical isomorphism
$$
H^2(C,\mu_2) \cong \mu_2 \cong \frac12\ZZ / \ZZ.
$$
Let us first check that $\{ Bh \}$ coincides with the image of the restriction $\alpha_{0|C} \in H^2(C,\mu_2)$ under this isomorphism.
Indeed, this follows immediately from the commutative diagram
$$
\xymatrix{
H^2(S,\frac12\ZZ) \ar[r] \ar[d] & H^2(C,\frac12\ZZ) \ar[d] \ar@{=}[r] & \frac12\ZZ \ar[d]^{\{-\}} \\
H^2(S,\mu_2) \ar[r] & H^2(C,\mu_2) \ar@{=}[r] & \frac12\ZZ / \ZZ
}
$$
in which the vertical arrows are induced by the map $\exp(2\pi i(-)) : \frac12\ZZ \to \mu_2$,
and the horizontal arrows are given by the restriction to $C$.

So, it remains to check that $\alpha_{0|C}$ equals $\left\{ \frac12\deg\det E \right\}$.
For this we consider exact sequence of groups
$$
\xymatrix@1{ 1 \ar[rr] && \mu_2 \ar[rr] && \GL_2 \ar[rr]^-{(p,\det)} && \PGL_2\times G_m \ar[rr] && 1},
$$
where $p:\GL_2 \to \PGL_2$ is the canonical projection. From this we obtain an exact sequence of cohomologies
$$
H^1(C,\GL_2(\CO_C)) \to H^1(C,\PGL_2(\CO_C)) \oplus H^1(C,\CO^*_C) \to H^2(C,\mu_C)
$$
which can be rewritten as the following commutative (up to a sign which in $H^2(C,\mu_C) = \frac12\ZZ / \ZZ$ is insignificant) diagram
$$
\xymatrix{
H^1(C,\GL_2(\CO_C)) \ar[r]^-\det \ar[d]_p & H^1(C,\CO^*_C) \ar[d]^{\{\frac12\deg(-)\}} \\
H^1(C,\PGL_2(\CO_C)) \ar[r] & H^2(C,\mu_C)
}
$$
By definition, $\alpha_{0|C}$ comes from the class in $H^1(C,\PGL_2(\CO_C))$ of the restriction
of the Azumaya algebra~$\CB_{|C}$ and the class of $E$ in $H^1(C,\GL_2(\CO_C))$ is its lift.
On the other hand, the top horizontal and the right vertical arrows take $E$ precisely to $\{\frac12\deg\det E\}$.
\end{proof}

Consider the bounded derived category $\D^b(S,\CB)$ of coherent sheaves of $\CB$-modules on the surface $S$.
Its Grothendieck group comes with the Euler bilinear form on it
$$
\chi_{(S,\CB)}([F],[G]) = \sum (-1)^i \dim\Ext^i(F,G).
$$
We denote by $K_0(S,\CB)$ the quotient of the Grothendieck group by the kernel of the Euler form,
i.e. the {\em numerical Grothendieck group}.

\begin{lemma}\label{k0}
Assume that $\{Bh\} = \{ B^2 \} = \frac12$. Then there is no such pair of vectors $v_1,v_2 \in K_0(S,\CB)$
for which $\chi_{(S,\CB)}(v_1,v_2) = 1$, $\chi_{(S,\CB)}(v_2,v_2) = 0$.
\end{lemma}
\begin{proof}
Recall that $h \in H^{1,1}(S,\ZZ)$ denotes the positive generator of $\Pic S$ and
let $p \in H^4(S,\ZZ)$ be the class of a point.
By~\cite{HS1}, Proposition~1.2, there exists a linear map (the {\em $B$-twisted Chern character})
$$
\ch^B:K_0(S,\CB) \to H^\bullet(X,\QQ),
$$
such that:
\begin{enumerate}
\item $\Im \ch^B = H^{*,*}(S,B,\ZZ) := \exp(B)(H^0(S,\QQ) \oplus H^{1,1}(S,\QQ) \oplus H^4(S,\QQ)) \cap H^\bullet(S,\ZZ)$
(Corollary~2.4 and Remark~1.3~(ii));
\medskip
\item $\chi_{(S,\CB)}(F,G) = \langle \ch^B(F)\sqrt{\td(S)}, \ch^B(G)\sqrt{\td(S)} \rangle$,
where $\langle -,- \rangle$ is the Mukai pairing
$$
\langle r_1 + d_1 h + s_1 p , r_2 + d_2 h + s_2 p \rangle = r_1s_2 - 2d_1d_2 + s_1r_2,
\qquad
r_i,d_i,s_i \in \ZZ.
$$
\end{enumerate}

Let us describe $H^{*,*}(S,B,\ZZ)$.
We have $H^0(S,\QQ) \oplus H^{1,1}(S,\QQ) \oplus H^4(S,\QQ) = \{ r + dh + sp\ |\ r,d,s \in \QQ \}$
since $H^{1,1}(S,\QQ) = \QQ h$.
Further
$$
\exp(B)(r + dh + sp) = r + (rB + dh) + (rB^2/2 + dBh + s)p.
$$
So, to obtain an element of $H^{*,*}(S,B,\ZZ)$ we must have
$$
r \in \ZZ,\qquad
rB + dh \in H^2(S,\ZZ),\qquad
rB^2/2 + dBh + s \in \ZZ.
$$
Let us show that $r$ is even. Indeed, for $r$ odd $rB + dh \in H^2(S,\ZZ)$ would imply
$B + dh \in H^\bullet(S,\ZZ)$, hence the image of $B$ in $H^2(S,\CO_S)/H^2(S,\ZZ) = H^2(S,\CO_S^*)$ would be zero.
So, $r$ is even. Therefore $rB \in H^2(S,\ZZ)$, hence $dh \in H^2(S,\ZZ)$, so $d \in \ZZ$.
We conclude that $H^{*,*}(S,B,\ZZ)$ is generated by elements $2+2B$, $h$ and $p$,
and it is easy to see that the matrix of the bilinear form $\chi_{(S,\CB)}$ in this basis is
$$
\left(\begin{matrix} 8 - 4B^2 & -2Bh & 2 \\ -2Bh & -2 & 0 \\ 2 & 0 & 0 \end{matrix}\right).
$$
Note that the only potentially odd integer in the matrix is $-2Bh$, all the rest are definitely even.
So, if we have $\chi_{(S,\CB)}(v_1,v_2) = 1$ then the coefficient of $v_2$ at $h$
(in the decomposition of $v_2$ with respect to the above basis of $K_0(S,\CB)$) is odd.
Thus, it suffices to check that for any $v_2 \in K_0(S,\CB)$ such that $\chi_{(S,\CB)}(v_2,v_2) = 0$
the coefficient of $v_2$ at $h$ is even.

Indeed, assume that $v_2 = x(2+2B) + yh +zp$. Then
$$
0 = \chi_{(S,\CB)}(v_2,v_2) = (8-4B^2)x^2 -4Bh xy + 4xz -2y^2
$$
implies $y^2 = 4x^2 + 2xz - x(2B^2x - 2Bhy)$. Taking into account that $2B^2 \equiv 2Bh \equiv 1 \bmod 2$
we deduce that $y^2 \equiv x(x+y) \bmod 2$. It is clear that for $y \equiv 1 \bmod 2$
this has no solutions, so we conclude that $y$ should be even.
%
\end{proof}

In what follows we check that for the Azumaya algebra $\CB$ on a K3 surface $S$ arising
from a cubic fourfold $Y$ containing a plane the conditions of Lemma~\ref{k0}
are satisfied. Then it will follow that $\D^b(S,\CB) \not\cong \D^b(S')$ for any surface $S'$.
We will use freely the notation introduced in Section~\ref{secplane}.

\begin{lemma}\label{cbb}
We have $\{Bh\} = \{B^2\} = \frac12$.
\end{lemma}
\begin{proof}
Let us start with a computation of $\{ Bh \}$.
Recall that we have a fibration in 2-dimensional quadrics $\pi:\TY \to \PP(B)$ and $\CM$ is the Hilbert scheme
of lines in fibers of $\pi$. Let $L$ be a generic line in $\PP(B)$ and $C = L\times_{\PP(B)} S = f^{-1}(L)$.
Restricting the fibration $\pi:\TY \to \PP(B)$ to $L$ we obtain a pencil of quadrics $\{Q_\lambda\}_{\lambda \in L}$.
Then $\CM_C = \CM\times_S C$ is the Hilbert scheme of lines in $Q_\lambda$.
By Lemma~\ref{bh} it suffices to check that $\CM_C \cong \PP_C(E)$ with $\deg\det E$ being odd.

Consider each $Q_\lambda$ as a quadric in $\PP(V)$.
It intersects with $\BP = \PP(A)$ in a conic $q_\lambda$, so we have a pencil of conics $\{q_\lambda\}_{\lambda \in L}$ in $\BP$.
It is clear that for sufficiently general $L$ the following two conditions are satisfied:
\begin{enumerate}
\item the base locus of the pencil $\{q_\lambda\}$ on $\BP$ is a quadruple of distinct points such that
any triple of them is noncollinear;
\item the points $\lambda'_1,\dots,\lambda'_6 \in L$ corresponding to singular quadrics $Q_\lambda$ with $\lambda \in L$
(and hence to the ramification points of $f:C \to L$) are pairwise distinct from the points $\lambda_1,\lambda_2,\lambda_3 \in L$
corresponding to reducible conics~$q_\lambda$.
\end{enumerate}

We will assume that for the chosen $L$ both properties $(1)$ and $(2)$ are satisfied.

Our main observation is that each line on a quadric $Q_\lambda$ intersects the conic $q_\lambda = Q_\lambda \cap \BP$ in a unique point, hence the Hilbert scheme
of lines on $Q_\lambda$ identifies either with
$$
\begin{cases}
q_\lambda\sqcup q_\lambda, & \text{if $q_\lambda$ is irreducible and $Q_\lambda$ is smooth
(i.e. $\lambda \not\in \{\lambda_1,\lambda_2,\lambda_3,\lambda'_1,\lambda'_2,\lambda'_3,\lambda'_4,\lambda'_5,\lambda'_6\}$);}\\
q_\lambda, & \text{if $q_\lambda$ is irreducible and $Q_\lambda$ is singular
(i.e. $\lambda \in \{\lambda'_1,\lambda'_2,\lambda'_3,\lambda'_4,\lambda'_5,\lambda'_6\}$);}\\
q_\lambda^+\sqcup q_\lambda^-, & \text{if $q_\lambda = q_\lambda^+\cup q_\lambda^-$ is reducible and $Q_\lambda$ is smooth
(i.e. $\lambda \in \{\lambda_1,\lambda_2,\lambda_3\}$).}
\end{cases}
$$
An immediate consequence is the following. The blowup $\TBP$ of $\BP$ in the base locus of $q_\lambda$
has a natural structure of a conic bundle over $L$ with three reducible fibers (over the points $\lambda_i$).
Let $\TBP_C = \TBP\times_L C$. This is a conic bundle over $C$ with 6 reducible fibers, over the preimages $\lambda^\pm_i \in C$
of points $\lambda_i \in L$. Then $\CM_C$ is obtained from $\TBP_C$ by contracting the components $q_{\lambda_i}^+$ in the fibers over $\lambda^+_i$
and $q_{\lambda_i}^-$ in the fibers over $\lambda^-_i$. It follows that a vector bundle $E$ on $C$ such that $\CM_C = \PP_C(E)$
can be obtained as follows.

First, consider the contraction of components $q_{\lambda_i}^+$ in the fibers of $\TBP$ over $\lambda_i \in L$.
We will obtain a $\PP^1$-fibration over $L$ which is isomorphic to $\PP_L(E_0)$ for some vector bundle $E_0$
on $L$ of rank $2$. Then $\PP_C(f^*E_0)$ is obtained from $\TBP_C$ by contracting the components $q_{\lambda_i}^+$
both in the fibers over $\lambda^+_i$ and in the fibers over $\lambda^-_i$, whereof it follows that $\CM_C$
is obtained from $\PP_C(f^*E_0)$ by simple Hecke transformations in the fibers over the points $\lambda^-_i$
(a simple Hecke transformation of a $\PP^1$-bundle is a a blow-up of a point followed by contraction of the proper
preimage of the fiber containing this point). So, $\PP_C(E)$ differs from $\PP_C(f^*E_0)$ by three simple Hecke transformations.
It remains to note that $\deg\det f^*E_0 = \deg f^*\det E_0$ is even and to use the well-known fact
that whenever $\PP_C(E)$ and $\PP_C(E')$ are related by a simple Hecke transformation the parity of $\deg\det E$
and $\deg\det E'$ is different, so after three Hecke transformations we obtain $\PP_C(E)$ with $\deg\det E$ being odd.

Now, to compute $\{B^2\}$ we are going to use Lemma~\ref{k0}. By this Lemma for any $\alpha$-twisted sheaf $F$ on $S$ of rank 2 we have
$\chi_{(S,\alpha)}(F,F) = \langle \ch^B(F)\sqrt{\td(S)},\ch^B(F)\sqrt{\td(S)} \rangle$.
On the other hand, for any $\alpha$-twisted sheaf of rank $2$ we have
$\ch^B(F) = 2+2B+dh+s$ for some $d,s \in \ZZ$, so
$$
\chi_{(S,\alpha)}(F,F) = 8 - 4B^2 + 4s  - 2d^2 - 4dBh.
$$
Note that $8 + 4s  - 2d^2 - 4dBh = 4s - 2d(d + 2Bh)$ is divisible by~$4$. Indeed,
$2d$ is divisible by~$4$ for even~$d$, while for $d$ odd $2(d + 2Bh)$ is divisible by $4$.
We conclude that $\chi_{(S,\alpha)}(F,F) \equiv 4B^2 \bmod 4$.

On the other hand, under identification of the category of sheaves of $\CB$-modules on $S$
with the category of $\alpha$-twisted sheaves, the sheaf $\CB$ corresponds to a certain rank $2$ twisted sheaf.
So, we conclude that $\chi_{(S,\CB)}(\CB,\CB) \equiv 4B^2 \bmod 4$.

Now consider the covering $f:S \to \PP(B)$. We have the pushforward and the pullback functors $f_*:\D^b(S,\CB) \to \D^b(\PP(B))$,
$f^*:\D^b(\PP(B)) \to \D^b(S,\CB)$. Note that $f^*\CO_{\PP(B)} \cong \CB$, so by adjunction
$$
\chi_{(S,\CB)}(\CB,\CB) =
\chi_{(S,\CB)}(f^*\CO_{\PP(B)},\CB) =
\chi_{\PP(B)}(\CO_{\PP(B)},f_*\CB) =
\chi_{\PP(B)}(\CO_{\PP(B)},\CB_0) =
\chi_{\PP(B)}(\CB_0).
$$
It remains to note that $\CB_0 \cong \CO_{\PP(B)} \oplus (\Lambda^2A \otimes \CO_{\PP(B)}(-1) \oplus A \otimes \CO_{\PP(B)}(-2)) \oplus \Lambda^3A \otimes \CO_{\PP(B)}(-3)$,
so $\chi_{\PP(B)}(\CB_0) = 2$, whence $\{B^2\} = 2/4 = 1/2$.
\end{proof}

Now we can give a proof of Proposition~\ref{dbsb}.
Combining Lemma~\ref{cbb} with Lemma~\ref{k0} we conclude that $K_0(S,\CB)$ does not contain a pair of
elements $(v_1,v_2)$ such that $\chi_{(S,\CB)}(v_1,v_2) = 1$, $\chi_{(S,\CB)}(v_2,v_2) = 0$.
On the other hand, for any surface $S'$ in $K_0(S')$ there is such a pair. Indeed, just take
$v_1 = [\CO_{S'}]$, $v_2 = [\CO_p]$ where $\CO_p$ is a structure sheaf of a point.
We conclude that $K_0(S,\CB) \not\cong K_0(S')$ which implies that $D^b(S,\CB) \not\cong D^b(S')$.

\end{document}